\documentclass[english, leqno]{article}
\usepackage[latin1]{inputenc}
\usepackage[T1]{fontenc}
\usepackage{babel,textcomp}
\usepackage{amssymb}
\usepackage{amsmath}
\usepackage{amsfonts}
\usepackage{graphicx, url}
\usepackage{makeidx}

\numberwithin{equation}{section} 
\numberwithin{figure}{section} 

\title{Environmental contours and optimal design}

\begin{document}

\maketitle

\author{Kristina Rognlien Dahl \thanks{Department of Mathematics, University of Oslo, Norway (kristrd@math.uio.no). This work has received funding from the Norwegian Research Council SCROLLER project, project number 299897.}}

\author{Arne Bang Huseby \thanks{Department of Mathematics, University of Oslo, Norway (arne@math.uio.no)}}

\begin{abstract} 
Classical environmental contours are used in structural design in order to obtain upper bounds on the failure probabilities of a large class of designs.  Buffered environmental contours, first introduced in \cite{DahlHuseby}, serve the same purpose, but with respect to the so-called buffered failure probability.  In contrast to classical environmental contours, buffered environmental contours do not just take into account failure vs. functioning, but also to which extent the system is failing.  This is important to take into account whenever the consequences of failure are relevant.  For instance, if we consider a power network, it is important to know not just that the power supply is failed, but how many consumers are affected by the failure.  In this paper, we study the connections between environmental contours, both classical and buffered, and optimal structural design.  We connect the classical environmental contours to the risk measure value-at-risk.  Similarly, the buffered environmental contours are naturally connected to the convex risk measure conditional value-at-risk.  We study the problem of minimizing the risk of the cost of building a particular design.  This problem is studied both for value-at-risk and conditional-value-at-risk.  By using the connection between value-at-risk and the classical environmental contours, we derive a representation of the design optimization problem expressed via the environmental contour.  A similar representation is derived by using the connection between conditional value-at-risk and the buffered environmental contour.  From these representations, we derive a sufficient condition which must hold for an optimal design.  This is done both in the classical and the buffered case.  Finally, we apply these results to solve a design optimization problem from structural reliability.
\end{abstract}



\def\mc#1{\mathcal #1}
\def\mb#1{\mathbb #1}
\def\bm#1{\mbox{\boldmath $#1$}}
\def\proof{{\parindent 0pt \textbf{Proof: }}}
\def\case#1{{\parindent 0pt \textsc{Case $#1$. }}}
\def\scenario#1{{\medskip\parindent 0pt \textsc{Scenario $#1$. }}}
\def\step#1{{\parindent 0pt \textsc{Step $#1$. }}}
\def\qed{\qquad \blacksquare}
\def\eqd{\stackrel{d}{=}}
\def\nsubseteq{{\not\subseteq}}
\def\I{\mathop{\rm \raisebox{0.1pt} I}}
\def\E{\mathop{\rm \raisebox{0.1pt} E}}
\def\Var{\mathop{\rm \raisebox{0.1pt} Var}}
\def\SD{\mathop{\rm \raisebox{0.1pt} SD}}
\def\Corr{\mathop{\rm \raisebox{0.1pt} Corr}}
\def\Cov{\mathop{\rm \raisebox{0.1pt} Cov}}
\def\IFF{\mbox{ if and only if }}
\def\AND{\mbox{ and }}
\def\FOR{\mbox{ for }}
\def\FORALL{\mbox{ for all }}
\def\Podist{\mathop{\rm Po}}
\def\Bindist{\mathop{\rm Bin}}
\def\Gamdist{\mathop{\rm Gamma}}
\def\Betadist{\mathop{\rm Beta}}
\def\Logdist{\mathop{\rm Lognorm}}
\def\imph{\emph}
\def\hangeq{\bigskip \hangindent=43pt}
\def\dg{^{\circ} {\rm C}}
\def\mt{{\rm m}}
\def\sp{{\,\,\,}}

\def\riscuegraphics#1{\includegraphics[width=6.8cm]{#1}}
\def\eqnref#1{(\ref{#1})}
\def\thmref#1{Theorem \ref{#1}}
\def\lemref#1{Lemma \ref{#1}}
\def\propref#1{Proposition \ref{#1}}
\def\corref#1{Corollary \ref{#1}}
\def\defnref#1{Definition \ref{#1}}
\def\exref#1{Example \ref{#1}}
\def\remref#1{Remark \ref{#1}}
\def\algref#1{Algorithm \ref{#1}}
\def\tblref#1{Table \ref{#1}}
\def\figref#1{Figure \ref{#1}}
\def\secref#1{Section \ref{#1}}
\def\subsecref#1{Subsection \ref{#1}}
\def\appref#1{Appendix \ref{#1}}

\def\captit#1{\caption{\textit{#1}}}

\newtheorem{theorem}{Theorem}[section]
\newtheorem{corollary}[theorem]{Corollary}
\newtheorem{definition}[theorem]{Definition}
\newtheorem{proposition}[theorem]{Proposition}
\newtheorem{lemma}[theorem]{Lemma}
\newtheorem{example}[theorem]{Example}
\newtheorem{remark}[theorem]{Remark}

\section{Introduction}

In this paper, we will consider the problem of design optimization. We will minimize the risk of the cost of a structural design. The cost of the structural design is composed of two parts: A fixed failure cost $K$ which occurs in case of system failure and a cost function $\kappa(\bm{x})$ which only depends on the chosen design $\bm{x}$. This risk-of-cost minimization will be done with respect to two different risk measures: Value-at-risk and conditional value-at-risk. Conditional value-at-risk is a convex risk measure, which takes into account not just whether a system functions or fails, but to which extent it fails. We connect the value-at-risk and conditional value-at-risk to environmental contours via functions $C(\bm{u})$ and $\bar{C}(\bm{u})$. This connection to the functions $C(\bm{u})$ and $\bar{C}(\bm{u})$ allows us to get an alternative characterisation of the risk-minimization problems.

The structure of the paper is as follows: In Section \ref{sec: var}, we recall the definition of value-at-risk and derive som properties of this risk measure. In Section \ref{sec: env_contour}, we recall the concept of environmental contours and buffered environmental contours. In Section \ref{sec: var_design}, we derive an alternative characterization of the design optimization problem of minimizing the value-at-risk of the cost of a structure by connecting this problem to environmental contours. In Section \ref{sec: example}, we apply this methodology to a structural design problem. Value-at-risk ignores the tail of the distribution of the structure function, therefore we recall the definition of another risk measure, conditional value-at-risk (CVaR), in Section \ref{sec: cvar}. We also derive some properties of CVaR. In Section \ref{sec: cvar_design}, we minimize the conditional value-at-risk of the cost of a structure. We derive an alternative characterization of this problem by connecting it to buffered environmental contours. Finally, in Section \ref{sec: u}, we discuss a criterion for selecting a set of initial design concepts related to the system of interest.

\section{Value-at-risk and some properties}
\label{sec: var}

Let $X$ be a random variable, representing risk. Define $S_X(x) := P(X > x)$. Let $\alpha \in (0,1)$ be a given probability representing an \emph{acceptable level of risk}. In the context of structural design, the value of $\alpha$ can for instance be determined by the firm based on the required return period of the system. See \cite{DahlHuseby} for further details. The $\alpha$-level \emph{value-at-risk} associated with the risk $X$, denoted by $V_{\alpha}[X]$, is given by $S^{-1}_X(\alpha)$. More formally, we define:
\begin{equation}
\label{eq:valueAtRisk}
V_{\alpha}[X] = S^{-1}_X(\alpha) = \inf\{x : P(X > x) \leq \alpha\}.
\end{equation}

value-at-risk is frequently used for risk management in banks and insurance companies. For more about value-at-risk as a risk measure, see e.g. \cite{jorion2000value} and \cite{best1998implementing}.

In the special case where $X$ is absolutely continuously distributed, we have:
$$
\begin{array}{lll}
V_{\alpha}[X] = S^{-1}_X(\alpha) = x \mbox{ }\IFF \\[\smallskipamount]
\hspace{3.5cm} P(X > x) = \alpha.
\end{array}
$$

More generally, if $S_X$ is strictly decreasing, we have that:
\begin{equation}
\label{eq:altValueAtRiskDecreasing}
\begin{array}{lll}
V_{\alpha}[X] = x \IFF \\[\smallskipamount]
\hspace{2cm} P(X > x) \leq \alpha \leq P(X \geq x).
\end{array}
\end{equation}

Finally, if $X$ is a \emph{discrete random variable}, we have that:
\begin{equation}
\label{eq:altValueAtRiskDiscrete}
\begin{array}{lll}
V_{\alpha}[X] = x \IFF \\[\smallskipamount]
\hspace{2cm} P(X > x) \leq \alpha < P(X \geq x).
\end{array}
\end{equation}

We now show some properties of value-at-risk which are needed to derive an alternative characterization of our design optimization problem in Section \ref{sec: var_design}.

\begin{theorem}[Monotone transform]
\label{prop: monotone}
For any strictly increasing continuous function $\phi: \mb{R}$ $\rightarrow$ $\mb{R}$ we have:
\begin{equation}
\label{eq:monotoneTransformProperty}
V_{\alpha}[\phi(X)] = S^{-1}_{\phi(X)}(\alpha) = \phi(S^{-1}_X(\alpha))
\end{equation}
\end{theorem}

\begin{proof}
We note that since $\phi$ is strictly increasing, it follows by (\ref{eq:valueAtRisk}) that:
\begin{align}
V_{\alpha}[\phi(X)] &= \inf\{y : P(\phi(X) > y) \leq \alpha\} \\
&= \inf\{y : P(X > \phi^{-1}(y)) \leq \alpha\}.
\end{align}
We then substitute $y = \phi(x)$ and $\phi^{-1}(y) = x$, and get:
\begin{align}
V_{\alpha}[\phi(X)] &= \inf\{\phi(x) : P(X > x) \leq \alpha\} \\
&= \phi(\inf\{x : P(X > x) \leq \alpha\}) \\
&= \phi(S^{-1}_X(\alpha)).
\end{align}

\hspace{10cm} $\square$
\end{proof}
\smallskip

Value-at-risk is linear, as shown in the following result.

\begin{corollary}[Linearity]
For $a > 0$ and $b \in \mathbb{R}$ we have:
$$
V_{\alpha}[a X + b] = a V_{\alpha}[X] + b.
$$ 
\end{corollary}

\begin{proof} 
The result follows directly from the monotonicity property by noting that:
$$
\phi(X) = a X + b
$$
is a strictly increasing function for all $a > 0$ and $b \in \mathbb{R}$.

\hspace{10cm} $\square$
\end{proof}

\section{Environmental contours}\footnote{This section is based on \cite{DahlHuseby}. We include it here for the sake of completeness.}
\label{sec: env_contour}

The classical approach to environmental contours was first introduced in \cite{Rosenblatt52}. A Monte Carlo approach to environmental contours was considered in \cite{HusebyVN-EnvCont-OE2013}, \cite{HusebyVN-EnvCont-SS2015} and \cite{HusebyVN-EnvCont-ESREL2014}.

In probabilistic structural design, it is common to define a \emph{performance function} $g(\bm{x}, \bm{V})$ depending on some deterministic design variables $\bm{x} = (x_1, x_2, \ldots, x_m)'$ representing various parameters such as capacity, thickness, strength etc. and some random \emph{environmental quantities} $\bm{V} = (V_1, V_2, \ldots, V_n)' \in \mc{V}$, where $\mc{V} \subseteq \mb{R}^n$.  The performance function is defined such that if $g(\bm{x}, \bm{V}) > 0$, the structure is \emph{failed}, while if $g(\bm{x}, \bm{V}) \leq 0$, the structure is \emph{functioning}. Moreover, for a given $\bm{x}$ the set $\mc{F} = \{\bm{v} \in \mc{V} : g(\bm{x}, \bm{V})> 0\}$ is called the \emph{failure region} of the structure.
An important part of the probabilistic design process is to make sure that $P(\bm{V} \in \mc{F})$ is acceptable for all failure regions $\mc{F}$ of interest, denoted  $\mc{E}$. 

In order to avoid failure regions with unacceptable probabilities, it is necessary to put some restrictions on the family of failure regions. This is done by introducing a set $\mc{B} \subseteq \mb{R}^n$ chosen so that for any relevant failure region $\mc{F}$ which do not overlap with $\mc{B}$, the failure probability $P(\bm{V} \in \mc{F})$ is \emph{small}. The family $\mc{E}$ is chosen relative to $\mc{B}$ so that $\mc{F} \cap \mc{B} \subseteq \partial\mc{B}$ for all $\mc{F} \in \mc{E}$, where $\partial\mc{B}$ denotes the boundary of $\mc{B}$. This boundary is then referred to as an \emph{environmental contour}. See \figref{fig:env_contour}.

\begin{figure}[h!]
\begin{center}
\includegraphics[width=0.5\textwidth]{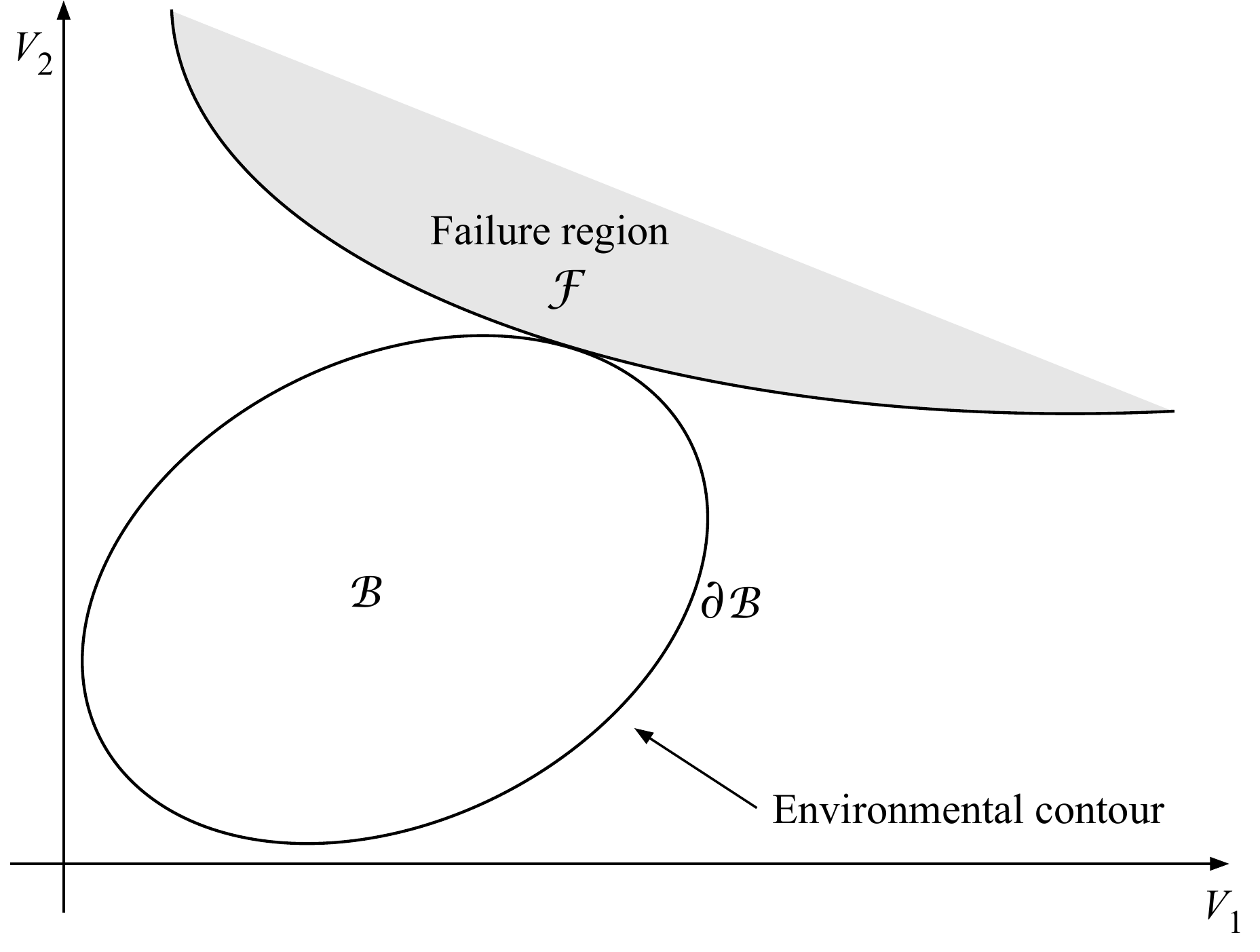}
\end{center}
\caption{An environmental contour $\partial \mc{B}$ and a failure region $\mc{F}$.}
\label{fig:env_contour}
\end{figure}

Following \cite{HusebyVE-EnvCont-ESREL2017} we define the \emph{exceedence probability} of $\mc{B}$ with respect to $\mc{E}$ as:
\begin{equation}
\label{eq:exceedenceProb}
P_{e}(\mc{B}, \mc{E}) := \sup \{p_f(\mc{F}) : \mc{F} \in \mc{E}\}.
\end{equation}
For a given \emph{target probability} $\alpha$ the objective is to choose an environmental contour $\partial\mc{B}$ such that:
$$
P_{e}(\mc{B}, \mc{E}) = \alpha
$$

The exceedence probability represents an upper bound on the failure probability of the structure assuming that the true failure region is a member of the family $\mc{E}$. Of particular interest are cases where one can argue that the failure region of a structure is \emph{convex}. That is, cases where $\mc{E}$ is the class of all convex sets which do not intersect with the interior of $\mc{B}$. 

\subsection{Monte Carlo contours}

There are many possible ways of constructing environmental contours. In this paper we connect the design optimization problem to the \emph{Monte Carlo based approach to environmental contours}, first introduced in \cite{HusebyVN-EnvCont-OE2013}, and improved in \cite{HusebyVN-EnvCont-SS2015} and \cite{HusebyVN-EnvCont-ESREL2014}.

Let $\mc{U}$ be the set of all unit vectors in $\mb{R}^n$, and let $\bm{u} \in \mc{U}$. We then introduce a function $C(\bm{u})$ defined for all $\bm{u} \in \mc{U}$ as: 
\begin{equation}
\label{eq:Cfunction}
C(\bm{u}) := \inf\{C : P(\bm{u}'\bm{V} > C) \leq \alpha\}
\end{equation}
Thus, $C(\bm{u})$ is the $(1-\alpha)$-quantile of the distribution of $\bm{u}'\bm{V}$. Given the distribution of $\bm{V}$, the function $C(\bm{u})$  can be estimated by using Monte Carlo simulation, see e.g. \cite{DahlHuseby}.

Then, from the previous definitions, 
\[
P[\bm{u}' \bm{V} > C(\bm{u})] = \alpha.
\]
We will use this equality to connect the optimal design problem to environmental contours, via the quantile function $C(\bm{u})$.

\subsection{Buffered environmental contours}

Similarly, so called \emph{buffered environmental contours}, first introduced in \cite{DahlHuseby}, can be estimated via a function 

\begin{equation}
\label{eq:bufCfunction}
\bar{C}(\bm{u}) := \E[\bm{u}' \bm{V} | \bm{u}' \bm{V} > C(\bm{u})].
\end{equation}

Buffered environmental contours are constructed similarly to classical environmental contours, with the exception that the failure probability of interest is the \emph{buffered failure probability}. For any probability level $\alpha$, the $\alpha$-\emph{superquantile} of $g(\bm{x},\bm{V})$, $\bar{q}_{\alpha}(\bm{x})$, is defined as:
\begin{equation}
\label{eq: superquantile}
\bar{q}_{\alpha}(\bm{x}) =  E[g(\bm{x}, \bm{V}) | g(\bm{x}, \bm{V}) > q_{\alpha}(\bm{x})].
\end{equation}
That is, the $\alpha$-superquantile is the conditional expectation of $g(\bm{x}, \bm{V})$ when we know that its value is greater than or equal the $\alpha$-quantile. Then, the buffered failure probability, $\bar{p}_f$, first introduced by Rockafellar and Royset \cite{RockafellarRoyset}, is defined as 
\begin{equation}
\label{eq: buffered_failure_prob}
\bar{p}_f(\bm{x}) = 1-\alpha,
\end{equation}
\noindent where $\alpha$ is chosen so that $\bar{q}_{\alpha}(\bm{x}) = 0$

From these definitions, it follows that 
\begin{equation}
\label{eq: alt_buffered_failure_prob}
\bar{p}_f(\bm{x})= P(g(\bm{x}, \bm{V}) > q_{\alpha}(\bm{x})) = 1-F(q_{\alpha}(\bm{x}))
\end{equation}
\noindent where $F$ denotes the distribution of the structure function $g$. Buffered environmental contours can be constructed via Monte Carlo similarly as classical contours. We will connect the design optimization problem wrt. conditional value-at-risk to buffered environmental contours in Section \ref{sec: cvar_design}.

\section{Value-at-risk and optimal design}
\label{sec: var_design}

We will now connect the optimal design problem with respect to value-at-risk to the quantile function $C(\bold{u})$. Then, we use this connection to derive an alternative characterization of the optimization problem. Some key references on design optimization and structural design are \cite{parkinson1993general} and \cite{cruse1997reliability}.

Let $\bm{V} = (V_1, \ldots, V_n) \in \mathcal{V}$ be a vector of environmental variables and let $\alpha \in (0,1)$ be a given probability representing an acceptable level of risk. We assume that we have determined a function $C(\bm{u})$ defined for all unit vectors $\bm{u} \in \mathbb{R}^n$ such that:
\begin{equation}
\label{eq:cfunction_1}
P[\bm{u}' \bm{V} > C(\bm{u})] = \alpha, \mbox{ for all } \bm{u} \in \mathbb{R}^n.
\end{equation}

We also introduce the following notation:
\begin{align*}
\Pi(\bm{u}) &= \{\bm{V} \in \mathcal{V} : \bm{u}' \bm{V} = C(\bm{u})\}, \\[2mm]
\Pi^+(\bm{u}) &= \{\bm{V} \in \mathcal{V} : \bm{u}' \bm{V} > C(\bm{u})\}, \\[2mm]
\Pi^-(\bm{u}) &= \{\bm{V} \in \mathcal{V} : \bm{u}' \bm{V} \leq C(\bm{u})\}
\end{align*}

Hence, we have:
\begin{equation}
\label{eq:cfunction_2}
\begin{array}{llll}
P[\bm{V} \in \Pi^+(\bm{u})] = P[\bm{u}' \bm{V} > C(\bm{u})] = \alpha, \\[\smallskipamount]
\hspace{4cm} \mbox{ for all } \bm{u} \in \mathbb{R}^n.
\end{array}
\end{equation}

\begin{remark}[Connection to MC contours]
Note that this is the same framework as what is frequently used in connection to Monte Carlo environmental contours, see Section \ref{sec: env_contour} as well as \cite{DahlHuseby}. The function $C$ corresponds to the quantile function used to construct environmental contours, see \eqref{eq:Cfunction}.
\end{remark}


Let the \emph{cost of system failure} be denoted by $K$. We introduce a deterministic function $\kappa = \kappa(\bm{x})$ representing the \emph{cost of the design} $\bm{x}$, and assume that:
$$ 
\kappa(\bm{x}) < K \mbox{ for all } \bm{x} \in \mathcal{X}.
$$
Note that this assumption implies that for any design of interest, system failure costs more than rebuilding the system. This means that system failure has other financial consequences than just having to rebuild the system. This will typically be the case in practise, for instance for telecommunication networks, subway networks or power production companies.

The \emph{total cost}, denoted $H$, is given by:
$$
H(\bm{V}, \bm{x}) = K \cdot \I[g(\bm{V}, \bm{x}) > 0] + \kappa(\bm{x}).
$$
\noindent where $\I[\cdot]$ denotes the indicator function. The $\alpha$-level value-at-risk of a given design, denoted $V_{\alpha}(H)$, is given by:
$$
V_{\alpha}(H) = S_{H}^{-1}(\alpha),
$$
where $S_H(h) = 1 - F_H(h) = P(H > h)$. Thus, $V_{\alpha}(H)$ is the $(1 - \alpha)$-percentile of the distribution of $H$.

Our \emph{main objective} is to choose a design $\bm{x}$ so that to minimize the value-at-risk of $H$, i.e.
$$
\min_{\bm{x} \in \mc{X}} V_{\alpha}\big(H(\bm{V}, \bm{x})\big)
$$
Since $\kappa(\bm{x})$ is deterministic, it follows by the linearity of $V_{\alpha}$ that:
$$
V_{\alpha}[H] = V_{\alpha}[K \cdot \I[g(\bm{V}, \bm{x}) > 0]] + \kappa(\bm{x}).
$$

We observe that $K \cdot \I[g(\bm{V}, \bm{x}) > 0]$ is a discrete random variable with only two possible values, $0$ and $K$. Its distribution is given by:
\begin{align*}
P[K \cdot \I[g(\bm{V}, \bm{x}) > 0] = K] &= P[g(\bm{V}, \bm{x}) > 0], \\[2mm]
P[K \cdot \I[g(\bm{V}, \bm{x}) > 0] = 0] &= P[g(\bm{V}, \bm{x}) \leq 0].
\end{align*}

\noindent By \eqref{eq:altValueAtRiskDiscrete} we know that:
$$
V_{\alpha}[K \cdot \I[g(\bm{V}, \bm{x}) > 0]] = y,
$$
if and only if:
\[
\begin{array}{lll}
P[K \cdot \I[g(\bm{V}, \bm{x}) > 0] > y] \leq \alpha \\[\smallskipamount]
\hspace{2cm} < P[K \cdot \I[g(\bm{V}, \bm{x}) > 0] \geq y]
\end{array}
\]

\noindent In particular, we have 
\[
P[K \cdot \I[g(\bm{V}, \bm{x}) > 0] > K] = 0 < \alpha.
\]
\noindent This implies that:
$$
V_{\alpha}[K \cdot \I[g(\bm{V}, \bm{x}) > 0]] = K,
$$
if and only if:
$$
P[K \cdot \I[g(\bm{V}, \bm{x}) > 0] \geq K] = P[g(\bm{V}, \bm{x}) > 0] > \alpha
$$

\noindent Furthermore, we have 
\[
P[K \cdot \I[g(\bm{V}, \bm{x}) > 0] \geq 0] = 1 > \alpha.
\]
\noindent This implies that:

$$
V_{\alpha}[K \cdot \I[g(\bm{V}, \bm{x}) > 0]] = 0,
$$
\noindent if and only if:
$$
\begin{array}{llll}
P[K \cdot \I[g(\bm{V}, \bm{x}) > 0] > 0] &=& P[g(\bm{V}, \bm{x}) > 0] \\[\smallskipamount]
&\leq& \alpha.
\end{array}
$$

\noindent Summarizing this, we get:
\begin{equation}
\label{eq: box}
\begin{array}{llll}
V_{\alpha}(K \cdot \I[g(\bm{V}, \bm{x}) > 0]) \\[\smallskipamount]
\hspace{2cm}= \begin{cases}
        K &\mbox{ if } P[g(\bm{V}, \bm{x}) > 0] > \alpha \\
       	0 &\mbox{ if } P[g(\bm{V}, \bm{x}) > 0] \leq \alpha
     \end{cases}
\end{array}
\end{equation}
\noindent From this it follows that:
\begin{equation}
\label{eq: star}
V_{\alpha}(H) = \begin{cases}
K + \kappa(\bm{x}) &\mbox{ if } P[g(\bm{V}, \bm{x}) > 0] > \alpha \\
\kappa(\bm{x}) &\mbox{ if } P[g(\bm{V}, \bm{x}) > 0] \leq \alpha
\end{cases}
\end{equation}
Since we have assumed that $\kappa(\bm{x}) < K$ for all $\bm{x} \in \mathcal{X}$, it follows that an optimal design $\bm{x}$ must be chosen so that:
\begin{equation}
\label{eq:riskcondition}
P[g(\bm{V}, \bm{x}) > 0] \leq \alpha
\end{equation}

\begin{theorem}[Halfspace condition VaR]
\label{thm:halfspace_condition}
A sufficient condition for \eqref{eq:riskcondition}	to hold is that $g(\bm{V}, \bm{x}) \leq 0$ for all $\bm{V}$ such that $\bm{u}'\bm{V} \leq C(\bm{u})$, where $\bm{u} \in \mathbb{R}^n$ is a suitably chosen unit vector.
\end{theorem}

\begin{proof} 
The condition implies that if $g(\bm{V}, \bm{x}) > 0$, then $\bm{u}'\bm{V} > C(\bm{u})$. Hence, by \eqref{eq:cfunction_1} we get that:
$$
P[g(\bm{V}, \bm{x}) > 0] \leq P[\bm{u}'\bm{V} > C(\bm{u})] = \alpha.
$$
Therefore, we conclude that inequality \eqref{eq:riskcondition} is satisfied. 

\hspace{10cm} $\square$
\end{proof}

We then let $\bm{u} \in \mathbb{R}^n$ be a unit vector and consider the following subclass of designs:
$$
\begin{array}{lll}
\mathcal{X}(\bm{u}) = \\[\smallskipamount]
\hspace{0.5cm} \{\bm{x} \in \mathcal{X} : g(\bm{V}, \bm{x}) \leq 0 \FORALL \bm{V} \in \Pi^-(\bm{u})\}.
\end{array}
$$
\noindent i.e., designs such that the systems functions for all $\bm{V} \in \Pi^-(\bm{u})$. By the halfspace condition, Theorem \ref{thm:halfspace_condition}, we know that condition \eqref{eq:riskcondition} is satisfied for all designs $\bm{x} \in \mathcal{X}(\bm{u})$. Hence, an optimal design within the subclass $\mathcal{X}(\bm{u})$ can be found by minimising $\kappa(\bm{x})$ with respect to $\bm{x} \in \mathcal{X}(\bm{u})$. Different choices of the unit vector $\bm{u}$ will generate different optimal designs. However, the choice of $\bm{u}$ may often be a result of initial concept decisions related to the system of interest. Thus, it may not be necessary to consider multiple subclasses of design.

\section{Example: Structural reliability}
\label{sec: example}

We consider a system whose performance depends on the non-negative environmental variables, $\bm{V} = (V_1, \ldots, V_n) \in \mathcal{V}$. The system fails if:
$$
A \bm{V} > \bm{x}
$$
where $A = A^{m \times n}$ is a matrix, and the design $\bm{x} = (x_1, \ldots, x_m)$ is a vector of \emph{strengths}.

The cost of the design $\bm{x}$ is given by:
$$
\kappa(\bm{x}) = c_1 x_1 + \cdots + c_m x_m.
$$

We want to minimize $\kappa(\bm{x})$ subject to 
\[
P[A \bm{V} > \bm{x}] \leq \alpha.
\] 

Since this failure probability may be difficult to compute, we instead minimize $\kappa(\bm{x})$ subject to:
\begin{equation}
\label{eq:optimalDesign}
\{\bm{V} \in \mathcal{V} : A \bm{V} > \bm{x}\} \subseteq \{\bm{V} \in \mathcal{V} : \bm{u}' \bm{V} > C(\bm{u})\}.
\end{equation}

It follows that if the design $\bm{x}$ satisfies \eqref{eq:optimalDesign}, then:
$$
P[A \bm{V} > \bm{x}] \leq P[ \bm{u}' \bm{V} > C(\bm{u})] = \alpha.
$$

For a given design $\bm{x}$, we can then check if it satisfies condition \eqref{eq:optimalDesign} by solving the following LP-problem:

\begin{equation}
\label{lp_problem}
\mbox{Minimise $\bm{u}' \bm{V}$ subject to $A \bm{V} \geq \bm{x}$.}
\end{equation}

Let $\bm{V}_0$ denote the solution to the minimization problem \eqref{lp_problem}. Then $\bm{x}$ satisfies condition \eqref{eq:optimalDesign} if and only if:
$$
\bm{u}' \bm{V}_0 > C(\bm{u}).
$$
By using a suitable iteration method one can then find a design $\bm{x}$ which minimizes $\kappa(\bm{x})$ subject to condition \eqref{eq:optimalDesign}.

\section{Conditional value-at-risk and some properties}
\label{sec: cvar}

So far, we have used value-at-risk as a design criterion. The problem with this is that VaR ignores the size of the outcomes in the tail of the distribution.

\begin{example}[Value-at-risk ignores the tail]

$\rm{VaR}_{0.05}(X)$ is the $x-$value such that only $5\%$ of the outcomes of $X$ are larger (i.e., worse in our context) than this value. Hence, $\rm{VaR}_{0.05}(X)$ ignores the size, and hence the consequences, of all values above this level.

\end{example}

Based on the previous definition of value-at-risk, conditional value-at-risk (CVaR),  denoted by $C_{\alpha}$, is defined as

\begin{equation}
\label{eq: CVaR}
C_{\alpha} (X) := \frac{1}{\alpha} \int_0^{\alpha} V_{u} (X) du
\end{equation}

That is, we compute the average of the value-at-risk in the $\alpha \%$ worst cases. Coherent risk measures, which conditional value-at-risk is an example of, were first introduced in \cite{artzner1999coherent}. CVaR is frequently used in mathematical finance, and to some extent in insurance mathematics. \cite{rockafellar2000optimization} and \cite{uryasev2000conditional} study optimization techniques in connection to CVaR.

Note that CVaR is also a convex risk measure, i.e.

\begin{enumerate}
 \item[$(i)$] (Convexity) For $0 \leq \lambda \leq 1$, $$C_{\alpha}(\lambda X + (1 - \lambda) Y) \leq \lambda C_{\alpha} (X) + (1 - \lambda) C_{\alpha} (Y).$$
 \item[$(ii)$] (Monotonicity) If $X \geq Y$, then \\
 $$C_{\alpha}(X) \geq C_{\alpha}(Y).$$
 \item[$(iii)$] (Translation invariance) If $m \in \mathbb{R}$, then $$C_{\alpha}(X + m \mbox{\bf{1}}) = C_{\alpha}(X) - m$$.
 \end{enumerate}
 
\begin{remark}
The monotonicity property is the other way around from what is common in financial mathematics because we view large positive values as bad (failure of system). In finance, greatly negative values are bad (losses). 
\end{remark}
 
\begin{theorem}
For any strictly increasing continuous function $\phi: \mb{R} \rightarrow \mb{R}$ we have:
\begin{equation}
\label{eq:monotoneTransformProperty_CVaR}
C_{\alpha}[\phi(X)] = \frac{1}{\alpha} \int_0^{\alpha} \phi(V_u(X)) du.
\end{equation}
\end{theorem}

\begin{proof} 

From the definition of CVaR \eqref{eq: CVaR}:
\begin{align*}
C_{\alpha}[\phi(X)] &= \frac{1}{\alpha} \int_0^{\alpha} S^{-1}_{\phi(X)}(u) du \\[\smallskipamount]
&= \frac{1}{\alpha} \int_0^{\alpha} \phi(S^{-1}_X(u)) du \\[\smallskipamount]
&= \frac{1}{\alpha} \int_0^{\alpha} \phi(V_u(X)) du
\end{align*}
\noindent where the second equality holds because of equation \eqref{eq:monotoneTransformProperty} of Proposition \ref{prop: monotone} for VaR. 

\hspace{10cm} $\square$
\end{proof}
\smallskip

In order to prove a monotone transform property of conditional value-at-risk, we need the following well-known inequality, included here for the sake of completeness:

\begin{theorem}[Jensen's inequality]
Let $(\Omega, \mathcal{F}, P)$ be a probability space. Let $g : \Omega \rightarrow \mathbb{R}$ be a $P$-integrable function. Also, assume that $\varphi : \mathbb{R} \rightarrow \mathbb{R}$ is a convex function. Then, 

\[
\varphi(\int_{\Omega} g(\omega) dP(\omega)) \leq \int_{\Omega} \varphi(g(\omega)) dP(\omega).
\] 

\end{theorem}

From Jensen's inequality, we find that for $f: [a,b] \rightarrow \mathbb{R}$, $\varphi : \mb{R} \rightarrow \mb{R}$ convex, we have

\[
\varphi \big( \frac{1}{b-a} \int_a^b f(x) dx \big) \leq \frac{1}{b-a} \int_a^b \varphi(f(x))dx.
\]

By using this, we can prove the following monotone transform property of CVaR:

\begin{theorem}[Monotone transform of CVaR]
Assume that $\phi : \mb{R} \rightarrow \mb{R}$ is a strictly increasing, continuous and convex function. Then,
\begin{equation}
\label{eq: Jensen}
\phi(C_{\alpha}[X]) \leq C_{\alpha}[\phi(X)].
\end{equation}
\end{theorem}

\begin{proof}
\[
\begin{array}{llll}
\phi(C_{\alpha}[X]) &=& \phi(\frac{1}{\alpha} \int_0^{\alpha} S^{-1}_X(u)) \\[\medskipamount]
&\leq& \frac{1}{\alpha} \int_0^{\alpha} \phi(S^{-1}_X(u)) du \\[\medskipamount]
&=& \frac{1}{\alpha} \int_0^{\alpha}  S^{-1}_{ \phi(X)}(u) du \\[\medskipamount]
&=& C_{\alpha}(\phi(X)).
\end{array}
\]
Here, the inequality holds from Jensen's inequality. The second to last equality follows because of equation \eqref{eq:monotoneTransformProperty} of the monotone transform proposition for VaR.

\hspace{10cm} $\square$
\end{proof}
\smallskip

Conditional value-at-risk is linear, as shown in the following result:

\begin{corollary}[Linearity of CVaR]
\label{cor: cvar_linear}
For $a > 0$ and $b \in \mathbb{R}$ we have:
$$
C_{\alpha}[a X + b] = a C_{\alpha}[X] + b.
$$ 
\end{corollary}

\bigskip

\begin{proof}
By using the definition of CVaR and the linearity of VaR, we see that
\[
\begin{array}{llll}
C_{\alpha}(aX + b) &=& \frac{1}{\alpha} \int_0^{\alpha} V_u(aX + b)du \\[\medskipamount]
&=& \frac{1}{\alpha} \int_0^{\alpha} \{ a V_u(X) + b\}du \\[\smallskipamount]
&=& a (\frac{1}{\alpha} \int_0^{\alpha} V_u(X)du) + b \\[\smallskipamount]
&=& aC_{\alpha}(X) + b.
\end{array}
\]

\hspace{10cm} $\square$
\end{proof}

\section{Conditional value-at-risk and optimal design}
\label{sec: cvar_design}

Parallel to the VaR-case, we would like to choose an optimal design $\bm{x}$ such that the conditional value at risk of the total cost is minimized:

\[
\min_{\bm{x} \in \mc{X}} C_{\alpha}(H(\bm{V}, \bm{x}))
\]
\noindent where, as before, $H(\bm{V}, \bm{x}) = K \cdot \I[g(\bm{V}, \bm{x}) > 0] + \kappa(\bm{x})$. From the linearity of CVaR (see Corollary \ref{cor: cvar_linear}),

\[
 C_{\alpha}(H) = K \cdot C_{\alpha}( \I[g(\bm{V}, \bm{x}) > 0]) + \kappa(\bm{x}).
\]

Note that $V_{u}$ is decreasing in $u$ from its definition. Also, note that



%

%


\begin{equation}
\label{eq: ineq}
\begin{array}{llll}
C_{\alpha}( \I[g(\bm{V}, \bm{x}) > 0]) \\[\smallskipamount]
\hspace{2cm} = \frac{1}{\alpha} \int_0^{\alpha} V_u( \I[g(\bm{V}, \bm{x}) > 0])du \\[\medskipamount]
\hspace{2cm} = \frac{1}{\alpha} \int_0^{\min\{P(g(\bm{V}, \bm{x}) > 0), \alpha \}} 1 du \\[\medskipamount]
\hspace{2cm} =  \frac{1}{\alpha} \min\{P(g(V,x) > 0), \alpha\} \\[\medskipamount]
\hspace{2cm} \geq V_{\alpha} ( \I[g(\bm{V}, \bm{x}) > 0]).
\end{array}
\end{equation}
Here, the second equality follows from \eqref{eq: box}-\eqref{eq: star}. The inequality follows from the formula for $V_{\alpha}( \I[g(\bm{V}, \bm{x}) > 0])$ in equation \eqref{eq: box}. Also, if $P(g(\bm{V}, \bm{x}) > 0) > \alpha$, we see that
\begin{equation}
\label{eq: equality}
\min\{P( g(\bm{V}, \bm{x}) > 0), \alpha \}=\alpha.
\end{equation}
\noindent Hence $C_{\alpha} = 1$  (the same as $V_{\alpha}$ in this case). The property in \eqref{eq: ineq} is also true in general: Conditional value-at-risk, $C_{\alpha}$, is more conservative than value-at-risk, $V_{\alpha}$. 

Now, consider two cases: Let case $1$ be the case where 
\[
P[g(\bm{V}, \bm{x}) > 0] > \alpha,
\]
\noindent and case $2$ be the case where 
\[
P[g(\bm{V}, \bm{x}) > 0] \leq \alpha.
\]
The calculations leading to \eqref{eq: ineq}-\eqref{eq: equality} imply that
\begin{equation}
\label{eq: final}
C_{\alpha}(H) = \begin{cases}
K + \kappa(\bm{x}) \mbox{ in case } 1 \\
K \frac{P[g(\bm{V}, \bm{x}) > 0] }{\alpha} + \kappa(\bm{x}) \mbox{ in case } 2. 
\end{cases}
\end{equation}
Note that $0 \leq \frac{P[g(\bm{V}, \bm{x}) > 0] }{\alpha} \leq 1$ in case 2 above (since $P[g(\bm{V}, \bm{x}) > 0] \leq \alpha$). Also, note that our assumption that $\kappa(\bm{x}) < K$ for all $\bm{x} \in \mathcal{X}$, is no longer enough to guarantee that the optimal design should be chosen such that $P[g(\bm{V}, \bm{x}) > 0] \leq \alpha$. 

By considering the difference between the two cases in equation \eqref{eq: final}, we find that a sufficient condition to ensure that the optimal design satisfies $P[g(\bm{V}, \bm{x}) > 0] \leq \alpha$ is:
\begin{equation}
\label{eq: sufficient}
\kappa(\bm{x}_2) - \kappa(\bm{x}_1) \leq \frac{K}{\alpha}(\alpha - P[g(\bm{V}, \bm{x}_2) > 0])
\end{equation}

\noindent for all $\bm{x}_2$ such that $P[g(\bm{V}, \bm{x}_2) > 0] \leq \alpha$ and $\bm{x}_1$ (that is, case $2$) such that $P[g(\bm{V}, \bm{x}_1) > 0] > \alpha$ (i.e., case $1$). Note that this slightly resembles a Lipschitz condition for the cost function $\kappa(\cdot)$. 

Assume, like before, that we have determined a function $C(\bm{u})$ defined for all unit vectors $\bm{u} \in \mathbb{R}^n$ such that \eqref{eq:Cfunction} holds. Now, define a function, $\bar{C}(\bm{u})$, as follows
\begin{equation}
\label{eq:bufCfunction}
\bar{C}(\bm{u}) := \E[\bm{u}' \bm{V} | \bm{u}' \bm{V} > C(\bm{u})].
\end{equation}

Furthermore, introduce the following notation:
\begin{align*}
\bar{\Pi}(\bm{u}) &= \{\bm{V} \in \mathcal{V} : \bm{u}' \bm{V} = \bar{C}(\bm{u})\}, \\[2mm]
\bar{\Pi}^+(\bm{u}) &= \{\bm{V} \in \mathcal{V} : \bm{u}' \bm{V} > \bar{C}(\bm{u})\}, \\[2mm]
\bar{\Pi}^-(\bm{u}) &= \{\bm{V} \in \mathcal{V} : \bm{u}' \bm{V} \leq \bar{C}(\bm{u})\}
\end{align*}

and define 
\begin{equation}
\label{eq: gamma}
\Gamma(\bm{u}, \bm{V}):= \bm{u}\cdot \bm{V} - \bar{C}(\bm{u}). 
\end{equation}

\begin{remark}[Connection to buffered contours]
Note that this is the same framework as what is used in connection to buffered environmental contours, see \cite{DahlHuseby}. The function $\bar{C}$ corresponds to the superquantile function used to construct buffered environmental contours, see \eqref{eq:Cfunction}.
\end{remark}


For a fixed (but arbitrary) unit vector $\bm{u}$, let $\bar{\mathcal{X}}(\bm{u})$ denote the set of designs $\bm{x}$ such that $g(\cdot, \bm{x})$ dominated by $\Gamma(\bm{u}, \cdot)$. Then, for any $\bm{x} \in \bar{\mathcal{X}}(\bm{u})$,

\[
\begin{array}{llll}
P(g(\bm{V}, \bm{x}) > 0) &\leq& P(\Gamma(\bm{u}, \bm{x}) > 0) \\[\smallskipamount]
&=& P(\bm{u} \cdot \bm{V} - \bar{C}(\bm{u}) > 0) \\[\smallskipamount]
&=& P(\bm{u} \cdot \bm{V}>\bar{C}(\bm{u}))  \\[\smallskipamount]
&=& P(\bar{\Pi}^+(\bm{u})) \\[\smallskipamount]
&\leq& P(\Pi^+(\bm{u})) \\[\smallskipamount] 
&=& \alpha.
\end{array}
\]
\noindent where the last inequality follows because $\bar{C}(\bm{u}) > C(\bm{u})$, so by the definitions of $\bar{\Pi}^+(\bm{u})$ and $\Pi^+(\bm{u})$, we find that $\bar{\Pi}^+(\bm{u}) \subseteq \Pi^+(\bm{u})$. Hence,

\[
P(\bar{\Pi}^+(\bm{u})) \leq P(\Pi^+(\bm{u})).
\]

Therefore, we have proved that if $\bm{x} \in \bar{\mathcal{X}}(\bm{u})$, then 
\begin{equation}
\label{eq:risk_min_const}
P(g(\bm{V}, \bm{x}) > 0) \leq \alpha.
\end{equation}

We summarize this in the following theorem.

\begin{theorem}[Domination condition CVaR]
\label{thm:domination_condition}
A sufficient condition for \eqref{eq:risk_min_const} to hold is that $g(\cdot, \bm{x})$ is dominated by a function $\Gamma(\bm{u}, \cdot)$ of the form \eqref{eq: gamma}, where $\bm{u} \in \mathbb{R}^n$ is a suitably chosen unit vector.
\end{theorem}

If condition \eqref{eq: sufficient} is satisfied, we know that the optimal design should be chosen such that equation \eqref{eq:risk_min_const} holds. Let $\bm{u} \in \mathbb{R}^n$ be a suitably chosen unit vector. By the domination condition for CVaR, Theorem \ref{thm:domination_condition}, we know that the condition \eqref{eq:risk_min_const} is satisfied for all designs $\bm{x} \in \bar{\mathcal{X}}(\bm{u})$. Hence, an optimal design is found by minimising 
\[
K \frac{P[g(\bm{V}, \bm{x}) > 0] }{\alpha} +\kappa(\bm{x})
\]

with respect to $\bm{x} \in \bar{\mathcal{X}}(\bm{u})$.

\section{Choosing the unit vector $u$}
\label{sec: u}

Different choices of the unit vector $\bm{u}$ will generate different optimal designs. The choice of $\bm{u}$ may often be a result of initial concept decisions related to the system. If a firm has $N$ different initial concepts, $\bm{u}_1, \ldots, \bm{u}_N$ under consideration, the minimization problem can be solved for each of these $\bm{u}_i$'s, $i=1, \ldots, N$. This results in $N$ potentially optimal designs $\bm{x}_1, \ldots, \bm{x}_N$.

To find the optimal concept, the firm can compare the objective function values, i.e. $V_{\alpha}(H(\bm{V}, \bm{x_i}))$ or $C_{\alpha}(H(\bm{V}, \bm{x_i}))$, $i=1, \ldots, N$, of these designs. Assume that for a fixed design $\bm{x}$, we know that the corresponding performance function $g(\cdot, \bm{x})$ is monotone in some $V_i$-component, $i=1, \ldots, n$. Then one should choose the unit vector $\bm{u}$ such that it "follows the monotonicity". That is, if $g$ is non-decreasing in $V_i$, so $V_i \leq \bar{V}_i$ implies that $g((V_i, \bm{V}), \bm{x}) \leq g((\bar{V}_i, \bm{V}),\bm{x})$, then $\bm{u}$ should be chosen such that $u_i \in (0,1)$\footnote{The notation $(V_i, \bm{V})$ means the vecor $\bm{V}$ where component $i$ is $V_i$}.

If $g$ is non-increasing in $V_i$, so $V_i \leq \bar{V}_i$ implies that $g((V_i, \bm{V}), \bm{x}) \geq g((\bar{V}_i, \bm{V}), \bm{x})$, then $u$ should be chosen such that $u_i \in (-1,0)$. 

We make the previous statement more precise: Consider the VaR case. The CVaR case is parallel. Assume that there exists $\bm{V}, V_i \leq \bar{V}_i$ where the system fails in $(\bar{V}_i, \bm{V})$, but functions in $(V_i, \bm{V})$. Note that this assumption is slightly stricter than $g$ being monotone in component $i$. It corresponds to monotonicity as well as criticality of the $i$'th environmental component. Also, assume for contradiction that $u_i \in (-1,0)$. 

By assumption, 
$$
g\big((V_i, \bm{V}), \bm{x}\big) \leq 0 \mbox{ and } g\big((\bar{V}_i, \bm{V}), \bm{x}\big) > 0.
$$
\noindent That is, the system fails in $(\bar{V}_i, \bm{V})$, but functions in $(V_i, \bm{V})$. There exists a vector $\bm{u}$ such that the (by scaling) unit vector $(u_i, \bm{u})$ satisfies $(\bar{V}_i, \bm{V}) \in \Pi^-((u_i, \bm{u}))$ and  $(V_i, \bm{V}) \in \Pi^+((u_i, \bm{u}))$. 

From the definitions of $\Pi^+((u_i, \bm{u}))$ and $\Pi^-((u_i, \bm{u}))$, this implies that the system should function in $(\bar{V}_i, \bm{V})$ and fail in $(\bar{V}_i, \bm{V})$. But this contradicts the assumption. Hence, choosing $u_i \in (-1,0)$ leads to a contradiction, so $u_i$ should be chosen in the only other way possible, namely such that $u_i \in (0,1)$. The arguments in the case where $g$ is non-increasing in $V_i$ is parallel.

\section{Conclusions and further work}

So far, we have minimized the risk of the cost of a structural design wrt. VaR and CVaR. 

An alternative design optimization problem is to minimize the expected cost under a risk constraint:

\begin{equation}
\label{eq: alt_prob}
\begin{array}{llll}
&\min E[H(\bm{x}, \bm{V})] \\[\medskipamount]
\mbox{such that} \\[\medskipamount]
&\mbox{risk}(g(\bm{x},\bm{V})) \leq \alpha.
\end{array}
\end{equation}

Here, the risk-function, which depends on the performance function of the system, can be either value-at-risk or conditional value-at-risk.

By looking at this optimization problem, the environmental contour becomes a representation of the constraint. This problem and its connection to environmental contours are the topic of future works.

\bibliography{ESREL2020}
\bibliographystyle{abbrv}

\end{document}